\documentclass{amsart}
\usepackage{amsmath}
\usepackage{amssymb}

\newtheorem{theorem}{Theorem}[section]
\newtheorem{lemma}[theorem]{Lemma}

\newtheorem{corollary}[theorem]{Corollary}
\theoremstyle{definition}

\newtheorem{conjecture}[theorem]{Conjecture}

\theoremstyle{remark}
\newtheorem{remark}[theorem]{Remark}

\numberwithin{equation}{section}

\newcommand{\mdeg}{\mathrm{mdeg}}
\newcommand{\Aut}{\mathrm{Aut}}
\newcommand{\Tame}{\mathrm{Tame}}

\begin{document}

\title{Tame automorphisms with multidegrees in the form of arithmetic progressions}

\author{Jiantao Li}
\address{School of Mathematics, Jilin university, 130012, Changchun, China} \email{jtlimath@gmail.com}

\author{Xiankun Du}
\address{School of Mathematics, Jilin university, 130012, Changchun, China} \email{duxk@jlu.edu.cn}

\date{\today}
\thanks{This research was partially supported by
NSF of China (No.11071097, No.11101176) and ``211 Project" and ``985 Project" of Jilin University.}

\begin{abstract}

Let $(a,a+d,a+2d)$ be an arithmetic progression of positive integers. The following statements are proved:
\begin{enumerate}
\item If $a\mid 2d$, then $(a, a+d, a+2d)\in\mdeg(\Tame(\mathbb{C}^3))$.
\item If $a\nmid 2d$, then, except for arithmetic progressions of the form $(4i,4i+ij,4i+2ij)$ with $i,j \in\mathbb{N}$ and $j$ is an odd number, $(a, a+d, a+2d)\notin\mdeg(\Tame(\mathbb{C}^3))$. We also related the exceptional unknown case to a conjecture of Jie-tai Yu, which concerns with the lower bound of the degree of the Poisson bracket of two polynomials.
\end{enumerate}
\end{abstract}

\maketitle


\section{Introduction}

Throughout this paper, let $F=(F_1,\dots,F_n): \mathbb{C}^n\rightarrow \mathbb{C}^n$ be a polynomial map on $\mathbb{C}^n$, where $\mathbb{C}$ denotes the complex field. Denote by $\mdeg F:=(\deg F_1,\ldots,\deg F_n)$ the {\it multidegree} of $F$. Denote by $\Aut (\mathbb{C}^n)$ the group of all polynomial automorphisms of $\mathbb{C}^n$ and by $\mdeg$ the mapping from the set of all polynomial maps into the set $\mathbb{N}^n$, here and throughout, $\mathbb{N}$ denotes the set of all positive integers.

A polynomial automorphism $F=(F_1,\ldots,F_n)$ of $\mathbb{C}^n$ is called {\it elementary} if $$F=(x_1,\ldots,x_{i-1}, \alpha x_i+f(x_1,\ldots,x_{i-1},x_{i+1},\ldots,x_n), x_{i+1},\ldots,x_n)$$
for $\alpha\in \mathbb{C}^*$. Denote by $\Tame (\mathbb{C}^n)$ the subgroup of $\Aut (\mathbb{C}^n)$ that is generated
by all elementary automorphisms. The elements of $\Tame (\mathbb{C}^n)$ are called {\it tame automorphisms}.  The classical Jung-van der Kulk theorem \cite{Jung, Kulk} showed that every polynomial automorphism of $\mathbb{C}^2$ is tame. For many years people believe that $\Aut (\mathbb{C}^n)$ is equal to $\Tame (\mathbb{C}^n)$. However, in 2004, Shestakov and Umirbaev \cite{SU1,SU2} proved the famous Nagata conjecture, that is, the Nagata automorphism on $\mathbb{C}^3$ is not tame.

The multidegree plays an important role in the description of polynomial automorphisms.  It follows from Jung-van der Kulk theorem that if $F=(F_1, F_2)\in\Aut (\mathbb{C}^2)$ then $\mdeg F=(\deg F_1, \deg F_2)$ is principal, that is, either $\deg F_1 | \deg F_2$ or $\deg F_2 | \deg F_1$. And the famous Jacobian conjecture is equivalent to the assert that if $(F_1, F_2)$ is a polynomial map satisfying the Jacobian condition, then $\mdeg F=(\deg F_1, \deg F_2)$ is principal \cite{Abhyankar}.
But it is difficult to describe the multidegrees of polynomial maps in higher dimensions, even in dimension three. Recently, Kara\'{s} presented a series of papers concerned with multidegrees of tame automorphisms in dimension three. It is shown in \cite{K1,K5}  that there is no tame automorphism of $\mathbb{C}^3$ with multidegree $(3,4,5)$ and $(4,5,6)$. In \cite{K2}, it is proved that $(p_1,p_2,d_3)\in \mdeg(\Tame(\mathbb{C}^3))$ if and only if $d_3\in p_1\mathbb{N}+p_2\mathbb{N}$, where $2<p_1<p_2$ are prime numbers. In \cite{K3,K4}, similar conclusions are given: $(3,d_2,d_3)\in \mdeg(\Tame(\mathbb{C}^3))$ if and only if $3 | d_2$ or $d_3\in 3\mathbb{N}+d_2\mathbb{N}$; $(d_1,d_2,d_3)\in \mdeg(\Tame(\mathbb{C}^3))$ if and only if $d_3\in d_1\mathbb{N}+d_2\mathbb{N}$, where $d_1, d_2$ are coprime odd numbers.

Let $(a,a+d,a+2d)$ be an arithmetic progression of positive integers. The following statements are proved in this paper.
\begin{enumerate}
\item If $a\mid 2d$, then $(a, a+d, a+2d)\in\mdeg(\Tame(\mathbb{C}^3))$.
\item If $a\nmid 2d$, then $(a, a+d, a+2d)\notin\mdeg(\Tame(\mathbb{C}^3))$ with the exceptional unknown case of the form $(4i,4i+ij,4i+2ij)$ with $i,j \in\mathbb{N}$ and $j$ is an odd number. We also related this unknown case to a conjecture of Jie-tai Yu, which concerns with the lower bound of the degree of the Poisson bracket of two polynomials.
\end{enumerate}

\section{Preliminaries}

Recall that a pair $f, g \in \mathbb{C}[x_1,\ldots,x_n]$ is called $*$-{\it reduced} in \cite{SU1,SU2} if
\begin{enumerate}
\item $f, g$ are algebraically independent;
\item $\bar{f}, \bar{g}$ are algebraically dependent, where $\bar{f}$ denotes the highest homogeneous component of $f$;
\item $\bar{f}\notin \langle \bar{g} \rangle$ and $\bar{g}\notin \langle \bar{f} \rangle$.
\end{enumerate}

The following inequality plays an important role in the proof of the Nagata conjecture in \cite{SU1,SU2} and is also essential in our proofs.

\begin{theorem}\emph{(\cite[Theorem 3]{SU1})}.\label{inequality}
Let $f, g \in \mathbb{C}[x_1,\ldots,x_n]$ be a $*$-reduced pair, and $G(x,y)\in \mathbb{C}[x,y]$ with $\deg_y G(x,y)=pq+r,\ 0\leq r<p$, where $p=\frac{\deg f}{\gcd(\deg f, \deg g)}$. Then
$$\deg G(f,g)\geq q(p\deg g-\deg f-\deg g+\deg [f,g])+r \deg g.$$
\end{theorem}

Note that $[f,g]$ means the Poisson bracket of $f$ and $g$:
$$[f,g]=\sum_{1\leq i<j\leq n}(\frac{\partial f}{\partial x_i}\frac{\partial g}{\partial x_j}-\frac{\partial f}{\partial x_j}\frac{\partial g}{\partial x_i})[x_i,x_j].$$
By definition $\deg [x_i,x_j]=2$ for $i\neq j$ and $\deg 0=-\infty$,
$$\deg [f,g]=\max_{1\leq i<j\leq n}\deg \{(\frac{\partial f}{\partial x_i}\frac{\partial g}{\partial x_j}-\frac{\partial f}{\partial x_j}\frac{\partial g}{\partial x_i})[x_i,x_j]\}.$$
It is shown in \cite{SU1} that $[f,g]=0$ if and only if $f,g$ are algebraically dependent. And if $f, g$ are algebraically independent, then
$$\deg [f,g]=2+\max_{1\leq i<j\leq n}\deg(\frac{\partial f}{\partial x_i}\frac{\partial g}{\partial x_j}-\frac{\partial f}{\partial x_j}\frac{\partial g}{\partial x_i})$$

\begin{remark}\label{inequality2}
It is easy to shown (see \cite{K1} for example) that Theorem \ref{inequality} is true even if $f,g$ just satisfy:
(1) $f, g$ are algebraically independent; (2) $\bar{f}\notin \langle \bar{g} \rangle$ and $\bar{g}\notin \langle \bar{f} \rangle$.
\end{remark}
\begin{theorem}{\rm (\cite[Theorem 2]{SU2})}.\label{theorem}
Let $F=(F_1,F_2,F_3)$ be a tame automorphism of $\mathbb{C}^3$. If $\deg F_1+\deg F_2+\deg F_3>3$, (that is, $F$ is not a linear automorphism), then $F$ admits either an elementary reduction or a reduction of types I-IV (see \cite[Definitions 1-4]{SU2}).
\end{theorem}

Recall that we say a polynomial
automorphism $F=(F_1,F_2,F_3)$ admits an {\it elementary reduction} if
there exists a polynomial $g\in \mathbb{C}[x,y]$ and a permutation
$\sigma$ of the set $\{1,2,3\}$ such that $\deg
(F_{\sigma(1)}-g(F_{\sigma(2)},F_{\sigma(3)}))<\deg F_{\sigma(1)}$.

\section{Main results}
Note that if $(F_1,F_2,F_3)$ with multidegree $(d_1,d_2,d_3)$ is a tame automorphism, then, after a permutation $\sigma$, $(F_{\sigma(1)},F_{\sigma(2)},F_{\sigma(3)})\in \Tame (\mathbb{C}^3)$. Thus, without loss of generality we can assume that $d_1\leq d_2\leq d_3$. Next, by \cite[Proposition 2.2]{K1} it follows that if $d_1 | d_2$ or $d_3$ is a linear combination of $d_1$ and $d_2$ with coefficients in $\mathbb{N}$, then $(d_1,d_2,d_3)\in \mdeg(\Tame (\mathbb{C}^3))$.

The task now is to show when does an arithmetic progression $(a, a+d, a+2d)$ belong to $\mdeg(\Tame(\mathbb{C}^3))$, where $a, d \in\mathbb{N}$.

\begin{lemma}\label{LemOnIn}
An arithmetic progression $(a, a+d, a+2d)$ satisfies $a\mid d$ or
$a+2d\in a\mathbb{N}+(a+d)\mathbb{N}$ if and only if $a\mid 2d$.
\end{lemma}
\begin{proof}
If $a+2d\in
a\mathbb{N}+(a+d)\mathbb{N}$, then there exist $i,j\in\mathbb{N}$
such that $a+2d=ai+(a+d)j$, whence $(i+j-1)a=(2-j)d$.
\begin{enumerate}
\item If $i+j>1$. Thus, $2-j>0$, $j=0$ or $1$. If $j=0$, then $2d=(i-1)a$. If $j=1$, then $d=ia$. In both case we get $a\mid 2d$.
\item Otherwise, if $i+j\leq1$, it is easy to see that $d=0$ and hence $a\mid 2d$. More precisely, if $i=j=0$, then $a=-2d$, whence $d=0$; if $i=0$, $j=1$ or $i=1$, $j=0$, it is trivial that $d=0$.
\end{enumerate}

Conversely, suppose that $a\mid 2d$. If $a$ is odd, then $a\mid d$. If $a$ is even, it follows from $a\mid 2d$ that $d=\frac a2 m$ for some $m\in\mathbb{N}$,
whence $a+2d=(m+1)a\in a\mathbb{N}$.
\end{proof}

\begin{lemma}\label{LemOnReductionIII}
To prove that $(a, a+d, a+2d)\notin\mdeg(\Tame(\mathbb{C}^3))$, it is enough to show that every polynomial automorphism $F$ with multidegree $(a, a+d, a+2d)$ does not admit any elementary reduction.
\end{lemma}
\begin{proof}
By \cite[Theorem 27]{K6}, to prove that there is no tame automorphism of $\mathbb{C}^3$ with multidegree $(d_1,d_2,d_3)$, it suffices to show that such a hypothetical automorphism admits neither an elementary reduction nor a reduction of type III.

If $F$ with multidegree $(a, a+d, a+2d)$ admits a reduction of type III, then by \cite[Definition 3]{SU2} there exists $n\in \mathbb{N}$ such that

\begin{equation*}
(1) \left\{ \begin{aligned}
         n<a&\leq\frac32n,\\
         a+d&=2n,\\
         a+2d&=3n,
 \end{aligned} \right.
 \text{\quad or (2)} \left\{ \begin{aligned}
         a&=\frac32n,\\
         a+d&=2n,\\
        5n/2<a+2d&\leq3n.
 \end{aligned} \right.
\end{equation*}
It follows from the last two equalities in (1) that $a=d=n$, which contradicts  $n<a\leq\frac32n$. It follows from the first two equalities in (2) that $a=\frac32n, d=\frac12n$, which contradicts  $5n/2<a+2d\leq3n$.
Therefore, $F$ admits no reduction of type III. Thus, to prove that $(a, a+d, a+2d)\notin\mdeg(\Tame(\mathbb{C}^3))$, it is enough to show that every polynomial automorphism $F$ with multidegree $(a, a+d, a+2d)$ does not admit any elementary reduction.
\end{proof}

We are now in a position to show our main results.

\begin{theorem}\label{main}
Let $(a,a+d,a+2d)$ be an arithmetic progression of positive integers.
\begin{enumerate}
\item If $a\mid 2d$, then $(a, a+d, a+2d)\in\mdeg(\Tame(\mathbb{C}^3))$.
\item If $a\nmid 2d$, then $(a, a+d, a+2d)\notin\mdeg(\Tame(\mathbb{C}^3))$, except for the case that $(4i,4i+ij,4i+2ij)$ with $i,j\in\mathbb{N}$ and $j$ is an odd number.
\end{enumerate}
\end{theorem}

\begin{proof}
(1) If $a\mid 2d$, then $a\mid d$ or
$a+2d\in a\mathbb{N}+(a+d)\mathbb{N}$ by Lemma \ref{LemOnIn}, and hence $(a, a+d, a+2d)\in\mdeg(\Tame(\mathbb{C}^3))$ by \cite[Proposition 2.2]{K1}.

(2) Now suppose that $a\nmid 2d$. Let $F=(F_1,F_2,F_3)$ be a polynomial automorphism with multidegree $(a, a+d, a+2d)$. Then by Lemma \ref{LemOnReductionIII}, it suffices to show that $F$ admits no elementary reduction.
 By Lemma \ref{LemOnIn}, the condition $a\nmid 2d$ implies that $a\nmid d$ and $a+2d\notin a\mathbb{N}+(a+d)\mathbb{N}$.

 Set $\gcd(a,a+d)=\gcd(a+d,a+2d)=\gcd(a,d)=b$. Write $a=b\bar{a},
 d=b\bar{d}$. Then $\bar{d}\geq1$ and we can claim that $\bar{a}\geq 3$, since \begin{enumerate}
                                                         \item[(i)] if $a$ is an odd number, it follows from $a\nmid 2d$ that $b\neq a$, whence $\bar{a}=\frac ab\geq3$.
                                                         \item[(ii)] if $a$ is an even number, it follows from $a\nmid 2d$ that $b\neq a$ and $b\neq \frac a2$, whence $\bar{a}\geq3$.
                                                       \end{enumerate}
Now the proof proceeds into three cases.

Case 1: If $F$ admits an elementary reduction of the form
$(F_1,F_2,F_3-g(F_1,F_2))$ such that $\deg(F_3-g(F_1,F_2))<\deg F_3$, then $\deg F_3=\deg g(F_1,F_2)$. Since $\gcd(a,a+d)=b$, we have $p=\frac{\deg F_1}{\gcd(\deg F_1,\deg F_2)}=\bar{a}$. Set $\deg_y g(x,y)=\bar{a}q+r,\ 0\leq r<\bar{a}$. Since $(F_1,F_2,F_3)$ is a polynomial automorphism, it follows that $F_i, F_j(i,j=1,2,3)$ are algebraically independent, and hence $\deg[F_i,F_j]\geq 2$. Moreover, $\bar{F_i}\notin \langle \bar{F_j} \rangle$ since otherwise we have $\deg F_i | \deg F_j$ that contradicts to the fact that $a\nmid (a+d)$ and $a+2d\notin a\mathbb{N}+(a+d)\mathbb{N}$. By Remark \ref{inequality2},
\begin{align*}
a+2d&=\deg F_3=\deg g(F_1,F_2)\\
   &\geq q(\bar{a}(a+d)-a-(a+d)+\deg[F_1,F_2])+r(a+d)\\
   &\geq q(3(a+d)-a-(a+d)+2)+r(a+d)\\
   &=q(a+2d+2)+r(a+d).
\end{align*}
Thus, $q=0$ and $r\leq1$. Write $g(F_1,F_2)=g_1(F_1)+g_2(F_1)F_2$. Since $a\mathbb{N}\cap (a\mathbb{N}+(a+d))=\emptyset$, we have $a+2d=\deg F_3=\deg g(F_1,F_2)\in a\mathbb{N}$ or $a\mathbb{N}+(a+d)$, which contradicts $a+2d\notin a\mathbb{N}+(a+d)\mathbb{N}$.

Case 2: If $F$ admits an elementary reduction of the form
$(F_1-g(F_2,F_3),F_2,F_3)$, then $\deg F_1=\deg g(F_2,F_3)$. Since $\gcd(a+d,a+2d)=b$, we have $p=\bar{a}+\bar{d}$. Set $\deg_y g(x,y)=(\bar{a}+\bar{d})q+r,\ 0\leq r<\bar{a}+\bar{d}$. Then
\begin{align*}
a&=\deg F_1=\deg g(F_2,F_3)\\
   &\geq q((\bar{a}+\bar{d})(a+2d)-(a+d)-(a+2d)+\deg[F_2,F_3])+r(a+2d)\\
   &\geq q(4(a+2d)-(a+d)-(a+2d)+2)+r(a+2d)\\
   &=q(2a+5d+2)+r(a+2d).
\end{align*}
Thus, $q=r=0$. Write $g(F_2,F_3)=g_1(F_2)$. Then $a=\deg F_1=\deg g_1(F_2)\in (a+d)\mathbb{N}$, contrary to $a<a+d$.

Case 3: If $F$ admits an elementary reduction of the form
$(F_1,F_2-g(F_1,F_3),F_3)$, then $\deg F_2=\deg g(F_1,F_3)$.
\begin{enumerate}
\item[(i)]  If $a$ is an odd number, then $\gcd(a,a+2d)=\gcd(a,d)$ and $p=\bar{a}$. Set $\deg_y g(x,y)=\bar{a}q+r,\ 0\leq r<\bar{a}$. Then
     \begin{align*}
     a+d&=\deg F_2=\deg g(F_1,F_3)\\
        &\geq q(\bar{a}(a+2d)-a-(a+2d)+\deg[F_1,F_3])+r(a+2d)\\
        &\geq q(3(a+2d)-a-(a+2d)+2)+r(a+2d)\\
        &=q(a+4d+2)+r(a+2d).
    \end{align*}
Thus, $q=r=0$. Write $g(F_1,F_3)=g_1(F_1)$, which implies that $a+d=\deg F_2=\deg g(F_1,F_3)=\deg g_1(F_1)\in a\mathbb{N}$, contrary to $a\nmid a+d$.
\item[(ii)] Now let $a$ be an even number. If we additionally assume that $a\nmid 4d$, then $\gcd(a,a+2d)=\gcd(a,2d)=2\gcd(\frac a2,d)\neq a,\frac a2$, whence $p=\frac a{\gcd(a,a+2d)}\neq 1, 2$. Therefore $p\geq3$. Set $\deg_y g(x,y)=pq+r,\ 0\leq r<p$. Then
     \begin{align*}
     a+d&=\deg F_2=\deg g(F_1,F_3)\\
        &\geq q(p(a+2d)-a-(a+2d)+\deg[F_1,F_3])+r(a+2d)\\
        &\geq q(3(a+2d)-a-(a+2d)+2)+r(a+2d)\\
        &=q(a+4d+2)+r(a+2d).
    \end{align*}
Thus, $q=r=0$. A same contradiction follows as in (a).\\
Moreover, if $a$ is an even number and $4\nmid a$, then the condition $a\nmid 2d$ forces $a\nmid 4d$. More precisely, if $a$ is even with $4\nmid a$, then $a=2k$ for some odd number $k$. If $a\mid 4d$, then $k\mid 2d$ and hence $k\mid d$. Thus, $a\mid 2d$, a contradiction. Therefore, the only unknown case left is $4\mid a$ and $a\mid 4d$, that is, $(4i,4i+ij,4i+2ij)$ with $i,j \in\mathbb{N}$. Moreover, the condition $a\nmid 2d$ forces $j$ to be an odd number.
\end{enumerate}

Thus, except for the case that $(4i,4i+ij,4i+2ij)$ with $i,j\in\mathbb{N}$ and $j$ is an odd number, any polynomial automorphism $F$ with multidegree $(a,a+d,a+2d)$ admits no elementary reduction, and consequently, $(a, a+d, a+2d)\notin\mdeg(\Tame(\mathbb{C}^3))$.
\end{proof}

Now the only unknown case left is to show whether there is a tame automorphism with multidegree $(4i,4i+ij,4i+2ij)$($i,j\in\mathbb{N}$ and $j$ is an odd number). By Lemma \ref{LemOnReductionIII}, to show such an automorphism does not exist, we just need to show it admits no elementary reduction.

\begin{theorem}\label{4i}
Let $F=(F_1,F_2,F_3)$ be a polynomial automorphism with multidegree $(4i,4i+ij,4i+2ij)$ that $i,j\in\mathbb{N}$ and $j$ is an odd number. If $\deg[F_1,F_3]>\deg F_1$, then $F$ admits no elementary reductions.
\end{theorem}

\begin{proof}
Proceeding as in the proof of Theorem \ref{main}, the proof falls into three parts.

(1) If $F$ admits an elementary reduction of the form
$(F_1,F_2,F_3-g(F_1,F_2))$, then $\deg F_3=\deg g(F_1,F_2)$. Since $\gcd(4i,4i+ij)=i\gcd(4,j)=i$, we have $p=\frac{\deg F_1}{\gcd(\deg F_1,\deg F_2)}=4$. Set $\deg_y g(x,y)=4q+r,\ 0\leq r<4$. Then
\begin{align*}
 4i+2ij&=\deg F_3=\deg g(F_1,F_2)\\
   &\geq q(4(4i+ij)-4i-(4i+ij)+\deg[F_1,F_2])+r(4i+ij)\\
   &\geq q(8i+3ij+2)+r(4i+ij).
\end{align*}
Thus, $q=0$ and $r\leq1$. Write $g(F_1,F_2)=g_1(F_1)+g_2(F_1)F_2$. However, the conditions $4i\mathbb{N}\cap ((4i+ij)+4i\mathbb{N})=\emptyset$ and $4i+2ij\notin 4i\mathbb{N}+(4i+ij)\mathbb{N}$ imply that $\deg F_3=\deg g(F_1,F_2)$ is impossible.

(2) If $F$ admits an elementary reduction of the form
$(F_1-g(F_2,F_3),F_2,F_3)$, then $\deg F_1=\deg g(F_2,F_3)$. Since $\gcd(4i+ij,4i+2ij)=i\gcd(4,j)=i$, $p=4+j$. Set $\deg_y g(x,y)=(4+j)q+r,\ 0\leq r<4+j$. Then
\begin{align*}
  4i&=\deg F_1=\deg g(F_2,F_3)\\
   &\geq q((4+j)(4i+2ij)-(4i+ij)-(4i+2ij)+\deg[F_2,F_3])+r(4i+2ij)\\
   &\geq q(8i+9ij+2ij^2+2)+r(4i+2ij).
\end{align*}
Thus, $q=r=0$. Write $g(F_2,F_3)=g_1(F_2)$. Then $4i=\deg F_1=\deg F_2\in(4i+ij)\mathbb{N}$, a contradiction.

(3) If $F$ admits an elementary reduction of the form
$(F_1,F_2-g(F_1,F_3),F_3)$, then $\deg F_2=\deg g(F_1,F_3)$. Since $\gcd(4i,4i+2ij)=2i\gcd(2,j)=2i$, $p=2$. Set $\deg_y g(x,y)=2q+r,\ 0\leq r<2$. Then \begin{align*}
      4i+ij&=\deg F_2=\deg g(F_1,F_3)\\
        &\geq q(2(4i+2ij)-4i-(4i+2ij)+\deg[F_1,F_3])+r(4i+2ij)\\
        &>q(2ij+4i)+r(4i+2ij).
     \end{align*}
Note that the last inequality follows from $\deg[F_1,F_3]>\deg F_1$.
Thus, $q=r=0$. Write $g(F_1,F_3)=g_1(F_1)$. Then we get a contradiction to $4i\nmid 4i+ij$.

Consequently, $F$ admits no elementary reductions.
\end{proof}

The task now is to ask what is the lower bound of $\deg[F_1,F_3]$, particularly, if $\deg[F_1,F_3]>\deg F_1$ for all polynomials satisfying $\deg F_1=4i$ and $\deg F_3=4i+2ij$, $i,j\in\mathbb{N}$ and $j$ is an odd number, then we can give a complete description of whether $(a,a+d,a+2d)\in \mdeg(\Tame(\mathbb{C}^3))$.  This question is closely related to a conjecture of Jie-Tai Yu.
\begin{conjecture}\cite{Yu}\label{Yu}
Let $f$ and $g$ be algebraically independent polynomials in $k[x_1,\ldots,x_n]$ such that the homogeneous components of maximal degree of $f$ and $g$ are algebraically
dependent, $f$ and $g$ generate their integral closures $C(f)$ and $C(g)$ in $k[x_1,\ldots,x_n]$, respectively, and neither $\deg f | \deg g$ nor $\deg g | \deg f$, then
$$\deg[f,g]>\min\{\deg(f), \deg(g)\}.$$
\end{conjecture}
Although Conjecture \ref{Yu} has some counterexamples in \cite{Yu}, it is still of great interest to find a meaningful
lower bound of $\deg[f,g]$, and such a bound will give a nice description of $\Tame(\mathbb{C}^n)$ and $\Aut(\mathbb{C}^n)$. In particular, if Conjecture \ref{Yu} is valid for $f,g$ with $\deg f=4i$ and $\deg g=4i+2ij$, then we can claim that $(a, a+d, a+2d)\in\mdeg(\Tame(\mathbb{C}^3))$ if and only if $a\nmid 2d$.

\begin{corollary}
{\rm(1)} Let $(d_1,d_2,d_3)$ be a sequence of continuous integers, then
 $(d_1,d_2,d_3) \in \mdeg (\Tame(\mathbb{C}^3))$ if and only if $d_1\leq 2$.

{\rm(2)} Let $(d_1,d_2,d_3)$ be a sequence of continuous odd numbers, then
 $(d_1,d_2,d_3)\in\mdeg(\Tame(\mathbb{C}^3))$ if and only if $d_1=1$.

{\rm(3)} Let $(d_1,d_2,d_3)$ be a sequence of continuous odd numbers, then
 $(d_1,d_2,d_3)\in\mdeg(\Tame(\mathbb{C}^3))$ if $d_1\leq 4$, $(d_1,d_2,d_3)\notin\mdeg(\Tame(\mathbb{C}^3))$ if $d_1>4$ and $d_1\neq8$. The only unknown case left is $(8,10,12)$.
\end{corollary}
\begin{proof}
(1) If $d_1\leq 2$, then $(d_1, d_1+1, d_1+2)\in \mdeg(\Tame(\mathbb{C}^3))$ by Lemma \ref{LemOnIn}. If $d_1\geq3$ but $d_1\neq4$, then by Theorem \ref{main} $(d_1, d_1+1, d_1+2)\notin \mdeg(\Tame(\mathbb{C}^3))$.  The only case left is $(4,5,6)$, which is proved by Kara\'{s} \cite{K5} that $(4,5,6)\notin \mdeg(\Tame(\mathbb{C}^3))$.

(2) It follows from Lemma \ref{LemOnIn} and Theorem \ref{main} that a sequence of continuous odd numbers $(d_1,d_2,d_3)\in \mdeg(\Tame(\mathbb{C}^3))$ if and only if $d_1=1$.

(3) It follows from Lemma \ref{LemOnIn} that $(2k, 2k+2, 2k+4)\in\mdeg(\Tame(\mathbb{C}^3))$ when $k\leq 2$. If $k>2$ but $k\neq4$, then by Theorem \ref{main} $(2k, 2k+2, 2k+4)\notin \mdeg(\Tame(\mathbb{C}^3))$.
The only unknown case left is $(8,10,12)$.
\end{proof}

\bibliographystyle{amsplain}

\end{document}